\documentclass[12pt,a4paper,reqno,oneside]{amsart}

\usepackage{amsfonts,amssymb,amsmath,mathtools,hyperref}
\usepackage[english]{babel}
\usepackage{bigints}
\usepackage{amssymb}
\usepackage{pdfsync}

\usepackage[all,error]{onlyamsmath}

\usepackage{  amsrefs, enumerate}
 \usepackage{  url}
\usepackage{mathrsfs}

\numberwithin{equation}{section}

\usepackage{relsize,bigints}

\usepackage{graphicx,tikz}

\usepackage{color}

\newtheorem{thm}{Theorem}[section]
\newtheorem{lemma}[thm]{Lemma}
\newtheorem{proposition}[thm]{Proposition}

\newcommand{\eps}{\varepsilon}

\DeclareMathAlphabet{\mathpzc}{OT1}{pzc}{m}{it}

\newcommand{\mquad}[1]{\qquad \text{#1} \qquad}
\newcommand{\mqquad}[1]{\quad \text{#1} \quad}

\newcommand{\e}{\mathrm e}
\newcommand{\ud }{\, \mathrm d}

\newcommand{\sem}[1]{\left ( \e^{t#1} \right)_{t \ge 0}}
\newcommand{\dom}[1]{\mathcal D(#1)}
\newcommand{\grae}{\lim_{\eps \to 0+}}
\newcommand{\lam}{\lambda}
\newcommand{\rla}{R_\lam}

\newcommand{\grat}{\lim_{t\to \infty}}
\newcommand{\lil}{\lim_{\lam \to 0}}

\newcommand{\rez}[1]{\left (\lam - #1\right )^{-1}}

\newcommand{\mc}{\mathcal}

\newcommand{\R}{\mathbb R}

\newcommand{\rod}[1]{\mbox{$\left (#1 \right )_{t\ge 0}$}} 
\newcommand{\slam}{\sqrt{2\lam}}

\usepackage{graphicx,tikz}
\usetikzlibrary{arrows}
\definecolor{ngreen}{HTML}{006400}

\title[]{A limit theorem for certain Feller semigroups, and the distribution of the local time for Brownian motion when it exits an interval}

\author{Adam Bobrowski} 
\address{Lublin University of Technology\\Department of Mathematics\\Nadbystrzycka 38A\\20-618 Lublin, Poland}
\author{Andrey Pilipenko}
 \address{University of Geneva\\ 
Section de math\'ematiques\\
UNI DUFOUR\\
24, rue du G\`en\`eral Dufour\\
Case postale 64\\
1211 Geneva 4, Switzerland}
\begin{document}

\newcommand{\ce}{\mathpzc c}
\newcommand{\ceep}{\mathpzc c_\eps}

\maketitle
\section{Introduction}

Let $a<0$ and $b>0$ be two real numbers and let $\mathfrak C[a,b]$ denote the space of continuous function on the interval $[a,b]$. Furthermore, let $\ce \colon \R \to \R^+$  be a continuous function such that \begin{equation}\label{1:0} \int_{-\infty}^\infty \ce(x)\ud x \eqqcolon \gamma < \infty \mqquad{and} \grae \eps^{-1} \ce (\eps^{-1}x) =0, x\not = 0.\end{equation} The operator $A_0$ defined by 
 $A_0f =\frac 12 f''$ on the 
domain  $\dom{A_0}$ composed of all twice continuously differentiable functions on $[a,b]$ (with one-sided derivatives at the ends) such that $f''(a)=f''(b)=0$, is well known to generate a Feller semigroup in $\mathfrak C[a,b]$; the related process is a Brownian motion on $[a,b]$ stopped at the first moment when it hits $a$ or $b$: 
\[ \e^{tA_0} f(x) = \mathsf E_x f(w(t\wedge \tau)), \qquad x\in [a,b], f \in \mathfrak C[a,b], t \ge  0 \]
where $w$ is a standard Brownian motion, $\tau$ is the first time $w$ equals $a$ or $b$, and $\mathsf E_x$ denotes expectation conditional on the process starting at $x$. 
 By the Phillips perturbation theorem the operators $A_\eps, \eps \in (0,1]$, defined on the same domain by 
\[ A_\eps f = \tfrac 12 f'' - \ce_\eps f,\]
where $\ceep (x) = \eps^{-1} \ce (\eps^{-1}x), x\in [a,b]$ are also Feller generators. In fact,  
the Feynmann--Kac formula provides the following explicit form of the corresponding semigroups: 
\begin{equation}\label{1:1} \e^{tA_\eps}f (x) = \mathsf E_x f(w(t\wedge \tau))\e^{-\int_0^{t\wedge \tau} \ceep (w(s))\ud s },  \, x\in [a,b], f \in \mathfrak C[a,b], t \ge  0. \end{equation}
The process generated by $A_\eps$ behaves like that generated by $A_0$ but can be killed and thus no longer defined from a random time on; $\ceep (x)$ is the intensity of killing at $x\in [a,b]$. 

Since, as $\eps \to 0$, $\ceep$ approximates a multiple of Dirac delta at $0$, there exists a limit of $\e^{tA_\eps}f(x)$, at least for each $x\in [a,b],t\ge 0$ and $f$ separately, and equals
\begin{equation}\label{1:2} \mathsf E_x f(w(t\wedge \tau))\e^{-\gamma L_0(t\wedge \tau)},  \qquad x\in [a,b], f \in \mathfrak C[a,b], t \ge  0, \end{equation}
where $L_0$ denotes the L\'evy local time of Brownian motion at $0$, see \cite{ito,karatzas,rosen}.

Our main goal in this paper is to use the Trotter--Kato--Neveu convergence theorem to prove that the semigroups $\sem{A_\eps}$ indeed converge, even in the strong sense, and to characterize the domain of the generator of the limit semigroup. As an application we provide a semigroup theoretic proof of the fact that  $L_0(\tau)$ is exponentially distributed; this is usually done by stochastic analysis (see formula 2.3.3 on p.~356 in \cite{borodin}, or formula 2.3.3 on p.~362 in the 2015 corrected printing of this monograph; see also \cite{karatzas}*{(4.19), p. 429}). The main theorem can be stated as follows. 

\begin{thm}\label{thm:1} Let $\dom{A}\subset \mathfrak C[a,b]$ be composed of functions $f$ that are twice continuously differentiable in either of the subintervals $[a,0]$ and $[0,b]$ separately (with one-sided derivatives at the ends) and such that 
\begin{equation}\label{1:4} f''(a) = f''(b)=0, f''(0+)=f''(0-) \mqquad{and} f'(0+) - f'(0-) = \gamma f(0). \end{equation}
The operator $A$ defined on this domain by $Af = \frac 12 f''$ is a Feller generator. Moreover,  
\[ \grae \e^{tA_\eps} f = \e^{tA}f , \qquad t \ge 0, f\in \mathfrak C[a,b],\]
and the limit is uniform with respect to $t $ in finite intervals. 
\end{thm}

Notably, the theorem describes a singular perturbation: as $\eps \to 0+$, the influence of the killing term in $A_\eps$ is stronger and stronger at $x=0$ and becomes negligible at other points.  As a result there is no killing term in $A$ itself, but the characterization of the domain of $A$ features an additional boundary condition, the one saying that $f'(0+) - f'(0-) = \gamma f(0)$; this boundary condition describes the mechanism in which the limit process is killed at $x=0$.

In Section \ref{sec:potm} we provide a proof of Theorem \ref{thm:1}.  Asymptotic behavior of $\sem{A}$ is discussed in Section \ref{sec:ass}, and it is there that the distribution of $L_0(\tau)$ is calculated. Additional results are presented in Section \ref{cito}.

\section{Proof of the main theorem}\label{sec:potm}

The Trotter--Kato--Neveu theorem (see \cites{abhn,kniga,knigazcup,goldstein,pazy}) states, roughly speaking, that semigroups converge whenever their resolvents converge. Hence, we start by calculating the resolvents of $A_\eps$s; in this part we are guided by the ideas contained in \cite{fellera3}  (see also \cite[p. 13]{fellera4}, where a reference is made to \cite[p. 483]{hillebook}).

\subsection{Resolvents of $A_\eps$s}\label{sec:tro}
Let $\lam>0$ and $\eps >0$  be fixed throughout Section \ref{sec:tro}. The resolvent of $A_\eps$ is build of two fundamental blocks: the solutions $k_\eps$ and $\ell_\eps$ of the following Cauchy problems: 
\begin{align*} k_\eps '' (x) &= 2(\lam +\ceep)k_\eps (x),\quad x\in [a,b], \qquad k_\eps(a)=0,k_\eps '(a)=1,\\
 \ell_\eps '' (x) &= 2(\lam +\ceep)\ell_\eps (x),\quad x\in [a,b], \qquad \ell_\eps (b)=0,\ell_\eps'(b)=-1;
\end{align*}
we note that dependence of $k_\eps$ and $\ell_\eps$ on $\lam$ is notationally suppressed. It is clear that $k_\eps$ and $\ell_\eps$ are solutions to these problems iff 
\begin{align*}
k_\eps (x) &= x- a + 2\int_a^x \int_a^y (\lam + \ceep (z))k_\eps (z) \ud z\ud y, \qquad x \in [a,b],\\
\ell_\eps (x) &= b -x + 2\int_x^b \int_y^b (\lam + \ceep (z))\ell_\eps (z) \ud z\ud y, \qquad \, x \in [a,b].
\end{align*}
and so the question of existence and uniqueness of solutions can be settled by examining the operators $T_\eps$ and $S_\eps$ that are defined by 
\begin{align*}
T_\eps f(x) &= x- a + 2\int_a^x \int_a^y (\lam + \ceep (z))f (z) \ud z\ud y, \qquad x \in [a,b], f \in \mathfrak C[a,b],\\
S_\eps f(x) &= b -x + 2\int_x^b \int_y^b (\lam + \ceep (z))f (z) \ud z\ud y, \qquad x \in [a,b], f\in \mathfrak C[a,b],
\end{align*}
and can be easily seen to map $\mathfrak C[a,b]$ into itself. In such a study the notion of Bielecki-type norm (see \cites{abielecki,jedenipol,edwards}) is highly instrumental.  We will use two families of such norms, indexed by $\omega>0$:
\[ \|f\|_\omega \coloneqq \max_{a\le x\le b}\e^{-\omega(x-a)}|f(x)|, \qquad \|f\|_\omega^\diamond \coloneqq \max_{a\le x\le b}\e^{\omega(x-b)}|f(x)|;\] 
it is quite obvious that they are equivalent to the standard maximum norm. 

\begin{lemma}\label{lem:1} For $f,g\in \mathfrak C[a,b]$ and $\omega >0$, we have 
\begin{align*}
\|T_\eps f - T_\eps g\|_\omega & \le 2(\tfrac \lam {\omega^2} + \tfrac \gamma \omega) \|f-g\|_\omega, \\
 \|S_\eps f - S_\eps g\|_\omega^\diamond & \le 2(\tfrac \lam {\omega^2} + \tfrac \gamma \omega) \|f-g\|_\omega^\diamond .\end{align*}
\end{lemma}
\begin{proof}To prove the first estimate we write, for $x\in [a,b]$,  
\[ \e^{-\omega(x-a)} |T_\eps f(x) - T_\eps g(x)| \le 2\e^{-\omega(x-a)}\int_a^x \int_a^y (\lam + \ceep (z))\e^{\omega(z-a)} \ud z\ud y\|f-g\|_\omega,\]
and note that $\lam \e^{-\omega(x-a)}\int_a^x \int_a^y \e^{\omega(z-a)} \ud z\ud y\le \frac \lambda {\omega^2}$ whereas 
\begin{align*} \e^{-\omega(x-a)}\int_a^x \int_a^y \ceep (z)\e^{\omega(z-a)} \ud z\ud y&\le \e^{-\omega(x-a)}\int_a^x  \e^{\omega(y-a)}\int_a^y \ceep (z) \ud z\ud y\\
&\le \gamma \e^{-\omega(x-a)}\int_a^x  \e^{\omega(y-a)}\ud y\le \tfrac\gamma\omega.
\end{align*} Since the second estimate is proved analogously, we omit the details. 
  \end{proof}

It is an immediate consequence of our lemma that, by considering $T_\eps$ with norm $\|\cdot\|_\omega$ with sufficiently large $\omega$ we make $T_\eps$ a contraction, thus establishing that its fixed point $k_\eps$ exists and is unique; an analogous statement concerns $\ell_\eps$. In Section \ref{sec:coa} we will make use of the fact that the constant on the right-hand side in the estimate of Lemma \ref{lem:1} does not depend on $\eps$. As for now, however, we focus on the following lemma, a stepping stone to the proof of our convergence theorem. 

\begin{lemma}\label{lem:2}The resolvent of $A_\eps$ is given by 
\begin{equation}\label{2:1} \rez{A_\eps} g = h_\eps + \tfrac {g(a)}{(\lam +\ceep(a))\ell_\eps (a)}\ell_\eps +  \tfrac {g(b)}{(\lam +\ceep(b))k_\eps (b)}k_\eps, \qquad g \in \mathfrak C[a,b], \end{equation}
where
\[ h_\eps (x) \coloneqq \tfrac{2\ell_\eps (x)}{\ell_\eps (a)} \int_a^x k_\eps (y) g(y) \ud y + \tfrac{2k_\eps (x)}{k_\eps (b)} \int_x^b \ell_\eps (y) g(y) \ud y, \qquad x \in [a,b];\] dependence of $h_\eps$ on $\lam$ is notationally suppressed and so is, to recall, the dependence of $k_\eps$ and $\ell_\eps$.  
\end{lemma}
\begin{proof} Let $\mc A_\eps$ be defined by $\mc A_\eps f = \frac 12 f'' - \ceep f$ on the domain equal to the subspace $\mathfrak C^2[a,b]\subset \mathfrak C[a,b]$ of twice continuously differentiable functions on $[a,b]$, so that $\mc A_\eps$ is an extension of $A_\eps$. A straightforward calculation shows that $h_\eps$ belongs to $\mathfrak C^2[a,b]$ and $\lam h_\eps - \mc A_\eps h_\eps = g$. In this calculation it is essential to note that $u_\eps \coloneqq k_\eps' \ell_\eps - k_\eps \ell_\eps'$ in fact does not depend on $x\in [a,b]$, and so $u_\eps (x) = u_\eps (a)=u_\eps (b) = \ell_\eps (a) = k_\eps (b)$; to see that $u$ is a constant function  simply check that $u_\eps' =0$.  Since, by definition, $k_\eps$ and $\ell_\eps$ are eigenvalues of $\mc A_\eps$ corresponding do $\lam$, we conclude that $f_\eps$ defined to be the right-hand side of \eqref{2:1} satisfies $\lam f_\eps - \mc A_\eps f_\eps = g$. 

This leaves us with showing that $f_\eps $ belongs to $\dom{A_\eps}$. But, by the already proved part, $h_\eps'' (a) = 2(\lam +\ceep)h_\eps (a) - 2g(a)= -2g(a)$. Similarly, 
$\ell_\eps'' (a) = 2(\lam +\ceep (a))\ell_\eps (a) $ and $k_\eps ''(a) = 2(\lam +\ceep(a))k_\eps (a) = 0$, resulting in $f''_\eps (a)=0$. A similar calculation shows that $f_\eps ''(b)=0$, thus completing the proof. \end{proof}

\subsection{Convergence of resolvents}\label{sec:coa}

It is clear from Lemma \ref{lem:2} that convergence of $\rez{A_\eps}$, as $\eps \to 0$, hinges on the convergence of $k_\eps$ and $\ell_\eps$. This is established in our next result. 

\begin{lemma}\label{lem:3} As $\eps \to 0$, $k_\eps $ converges, uniformly in $x\in [a,b]$,  to a unique function $k$ such that 
\begin{equation}\label{2:2} k(x) = x- a + 2\lam \int_a^x \int_a^y k(z)\ud z + 2\gamma k(0)x[x\ge 0], \qquad x\in [a,b],\end{equation}
where $[\cdot]$ is the Iverson bracket. Similarly, $\ell_\eps$ converges to a unique function $\ell$ such that
\[ \ell (x) = b- x + 2\lam \int_x^a \int_y^b \ell (z)\ud z - 2\gamma \ell(0)x[x\le 0], \qquad x\in [a,b].\]
\end{lemma}
\begin{proof} To prove the first statement, we start by noting that, by Lemma \ref{lem:1}, the operators $T_\eps, \eps \in (0,1]$ have a common contraction constant provided that we consider $\mathfrak C[a,b]$ with the Bielecki-type norm with sufficiently large $\omega$. 
Hence, by the well-known extension of the Banach Fixed Point Theorem (see e.g. \cites{jedenipol,zmarkusem2b}),
 to prove that fixed points $k_\eps$ of $T_\eps$ converge it suffices to check that there is a strong limit, say, $T$, of $T_\eps$:
 in such a scenario the limit of fixed points is the fixed point of the limit operator. In other words, we are left with proving that 
$\grae T_\eps f = Tf, f \in \mathfrak C[a,b]$ where
\[ Tf(x) \coloneqq  x- a + 2\lam \int_a^x \int_a^y f(z)\ud z + 2\gamma f(0)x[x\ge 0], \qquad x\in [a,b].\]

 To this end, we first write 
\[ \left |\int_a^0 \int_a^y \ceep (z) f(z) \ud z \ud y \right | \le  \int_a^0 \int_a^y \ceep (z) \ud z \ud y \|f\|\]
 and note that $\int_a^y \ceep (z) \ud z\le \gamma$ and $\grae \int_a^y \ceep (z)\ud z =0$ for $y<0$. Hence, by the Lebesgue dominated convergence theorem, both expressions in the display above converge to $0$, as $\eps \to 0$. This further reduces our aim to showing that $\int_0^x |\int_a^y \ceep (z) f(z) \ud z - \gamma f(0)|\ud y$ converges uniformly to $0$, or, simply that $\int_0^b |\int_a^y \ceep (z) f(z) \ud z - \gamma f(0)|\ud y$ converges to $0$. 
But, again, the integrands here are bounded by $2\gamma \|f\|$ and converge to $0$: indeed, for $y>0$, $\grae \int_a^y \ceep (z) f(z)\ud z  = \grae \int_{\eps^{-1}a}^{\eps^{-1}b} \ce (z) f(\eps z) \ud z=\gamma f(0)$ by the Lebesgue dominated convergence theorem.

We omit the proof of the second statement, because it is analogous to that of the first one. 
\end{proof}

Here is another characterization of $k$ and $\ell$ introduced in Lemma \ref{lem:3}. Relation \eqref{2:2} shows that $k$ is twice continuously differentiable and satisfies $k'' = 2\lam k$ in either of the intervals $[a,0]$ and $[0,b]$; therefore, since $k$ belongs to $\mathfrak C[a,b]$,  so does $k''$ and in particular $k''(0+) =k''(0-)$. Also, \[ k(a)=0, k'(a)=1 \mquad{and} k'(0+) - k'(0-)= 2\gamma k(0).\] Similarly, $\ell$ satisfies 
$\ell'' = 2\lam \ell$ in either of the intervals $[a,0]$ and $[0,b]$, and we have 
\[ \ell(b)=0, \ell'(b)=-1 \mquad{and} \ell'(0+) - \ell'(0-)= 2\gamma \ell(0).\]
This yields the following explicit form of $k$ and $\ell$:
\begin{align}
k(x) &= \tfrac 1{\slam}\sinh\slam (x-a) - \tfrac \gamma \lambda \sinh\slam a \sinh \slam x[x\ge 0], \nonumber \\
\ell (x) &= \tfrac 1{\slam}\sinh\slam (b-x) - \tfrac \gamma \lambda \sinh\slam b \sinh \slam x[x\le 0].\label{2:3}
 \end{align}
\subsection{Proof of Theorem \ref{thm:1}}\label{sec:proof}
It is clear from Lemma \ref{lem:2} and Lemma \ref{lem:3} that the strong limit $\rla\coloneqq \grae \rez{A_\eps}$ exists (the second part of assumption \eqref{1:0} is used here) and is given by
\begin{equation}\label{2:4} \rla g = h + \tfrac {g(a)}{ \lam\ell (a)}\ell  +  \tfrac {g(b)}{ \lam k (b)}k, \qquad g \in \mathfrak C[a,b], \end{equation}
where
\[ h (x) \coloneqq \tfrac{2\ell (x)}{\ell (a)} \int_a^x k (y) g(y) \ud y + \tfrac{2k (x)}{k(b)} \int_x^b \ell  (y) g(y) \ud y, \qquad x \in [a,b]. \]
A calculation similar to that presented in Lemma \ref{lem:2} shows that $\rla g$ belongs to the domain of $A$, as defined in Theorem \ref{thm:1}, and that  $\lam \rla g - A\rla g = g$, that is, that the resolvent equation $\lam f - Af =g$ has a solution for every $g\in \mathfrak C[a,b]$. Since $A$ is obviously densely defined and satisfies the positive-maximum principle, $A$ is a Feller generator (see e.g.  \cites{bass,kniga,ethier,kallenbergnew}). In particular, the solution to the resolvent equation is unique and we have $\rez{A}=\rla$. This means that $\grae \rez{A_\eps} = \rez{A}$ strongly,  completing the proof of our main result by the Trotter--Kato--Neveu theorem.

\section{Asymptotic behavior of $\sem{A}$ and the distribution of $L_0(\tau)$}\label{sec:ass}

On the one hand, Theorem \ref{thm:1} says that, for each $f\in \mathfrak C[a,b]$, $\e^{tA_\eps}f (x)$ converges to $\e^{tA}f(x)$ uniformly with respect to $x\in [a,b]$ and $t$ in finite intervals. On the other hand, as already mentioned, $\e^{tA_\eps} f(x), $ given by \eqref{1:1} converges to \eqref{1:2} for each $x\in [a,b], t \ge 0$. It follows that 
\[ \e^{tA}f(x) =  \mathsf E_x f(w(t\wedge \tau))\e^{-\gamma L_0(t\wedge \tau)},  \qquad x\in [a,b], f \in  \mathfrak C[a,b], t \ge  0. \]
We will use this relation to find the distribution of $L_0(\tau)$. 

To this end, we note that a Brownian motion started at a point $x\in \R$ exits any interval in which $x$ is contained in a finite time  with probability $1$. It follows that, almost surely, we have $\grat w(t\wedge \tau)=w(\tau)$, and so, by the Lebesgue dominated convergence theorem, \begin{equation}\label{3:1} \grat \e^{tA}f(x) = \mathsf E_x f(w(\tau)) \e^{-\gamma L_0(\tau)}, \qquad  x\in [a,b], f\in \mathfrak C[a,b].\end{equation}
But we also have a different way to find this limit: since the resolvent of the generator is the Laplace transform of the related semigroup, and the trajectories of the semigroup are continuous functions of $t$, we have
\[ \grat \e^{tA}f (x) =  \lil\lam \int_0^\infty \e^{-\lam t} \e^{tA}f(x)\ud t= \lil\lam \rez{A}f(x). \]
In this new form the limit can be easily calculated explicitly as follows. 

\begin{lemma}For each $x\in [a,b]$ and $f\in \mathfrak C[a,b]$,
\[ \lil \lam \rez{A}f(x) = f(a) \frac{b-x -2\gamma bx [x\le 0]}{b-a - 2\gamma ab} + f(b) \frac{x-a -2\gamma ax [x\ge 0]}{b-a - 2\gamma ab}.  \]
 \end{lemma}

\begin{proof}In Section \ref{sec:proof} we have established that $\rla $ of \eqref{2:4} is just another notation for $\rez{A}$, and so we need to find the limit, as $\lam \to 0$, of the right hand side of \eqref{2:4} (with $g$ replaced by $f$) multiplied by $\lam$.

To this end, we claim first that $\lil\lam h=0$ (we recall that, to shorten and simplify formulae, dependence of $h,k$ and $\ell$ on $\lam$ has been notationally suppressed). Indeed, since $\ell$ is a decreasing function, the first term in the definition of $h$ does not exceed $2\int_a^x k(y) f(y)\ud y$ and this can be further estimated by $2\|f\|\int_a^b k(y)\ud y$ because $k$ is positive. Moreover, by \eqref{2:3},
\begin{align*} \int_a^b k(y) \ud y 
&= \tfrac {\cosh \slam (b-a)-1}{2\lam} - \tfrac {2\gamma}{\slam} \sinh \slam a \tfrac{ \cosh \slam b - 1}{2\lam}\end{align*}
implying that $\lil \lam \int_a^b k(x)\ud x =0$. This shows that the first term in the definition of $h$, when multiplied by $\lam$, converges to $0$, as $\lam \to 0$. Since an analogous proof leads to the conclusion that the same is true about the second term, our claim is established. 

We are thus left with studying the limits $\lil k$ and $\lil \ell$.  This, however, is straightforward by \eqref{2:3}. To wit, for $x\in [a,b]$,  
\[ \lil k(x) = x-a -2\gamma ax [x\ge 0] \mqquad{and} \lil \ell(x) = b-x -2\gamma bx [x\le 0], \]
as desired. \end{proof}

The lemma reveals that, for  $x\in [a,b]$ and $ f\in \mathfrak C[a,b]$,
 \begin{equation} \label{3:2} \mathsf E_x f(w(\tau)) \e^{-\gamma L_0(\tau)}=  f(a) \ell_* (x) + f(b) k_* (x), \qquad x\in [a,b],
 \end{equation}
where 
\[ \ell_* (x) \coloneqq \frac{b-x -2\gamma bx [x\le 0]}{b-a - 2\gamma ab} \mqquad{and} k_*(x)\coloneqq \frac{x-a -2\gamma ax [x\ge 0]}{b-a - 2\gamma ab}.\] 
Thus, choosing  $f(y)=1, y\in [a,b]$  yields
 \begin{equation} \label{3:3} \mathsf E_x \e^{-\gamma L_0(\tau)}=  \ell_* (x) + k_* (x)= \tfrac{b-a -\gamma (a+b)x +\gamma (b-a)|x|}{b-a - 2\gamma ab}, \qquad x\in [a,b],
 \end{equation}
and in particular  
\[  \mathsf E_0 \e^{-\gamma L_0(\tau)} = \frac{b-a}{b-a - 2\gamma ab}= \frac 1{1 + \frac{2(-a)b}{b-a}\gamma}.\] 
Since $\gamma$ can be chosen to be any positive number, we have proved that the Laplace transform of $L_0(\tau)$ coincides with that of an exponential random variable with parameter $\tfrac{2(-a)b}{b-a}$. In the specific case of $a=-b$, the parameter equals $b$.

The other quantities featured in \eqref{3:2} can also be plausibly interpreted: for example, $\ell_*(x)$
is the probability that the process generated by $A$,  started at $x\in [a,b]$, will reach $a$ and will be captured there before being  captured at $b$ or killed at $0$. Also, $x\mapsto \mathsf E_x \e^{-\gamma L_0(\tau)}$, as calculated in \eqref{3:3}, is an affine function of $x$ in either of the subintervals $[a,0]$ and $[0,b]$, and can be written as follows:
\begin{align*}
\mathsf E_x \e^{-\gamma L_0(\tau)} &= \tfrac xb + \tfrac {b-x}b\tfrac{b-a}{b-a-2\gamma ab}, \qquad x \ge 0 \\
\mathsf E_x \e^{-\gamma L_0(\tau)} &= \tfrac xa + \tfrac {a-x}a\tfrac{b-a}{b-a-2\gamma ab}, \qquad x \le 0. 
 \end{align*}
The first of these relations is a reflection of the fact that there are two possibilities for a process starting at $x\in (0,b)$ to be stopped at the ends of $[a,b]$.  It can either reach $b$ directly, without going through zero, in which case the local time at $0$ accumulated up to $\tau$ is zero; the probability of this case is $\frac xb$.  With probability $\frac {b-x}b$, however, it can go through $0$ first, and the local time accumulated at $\tau$ is then $  \mathsf E_0 \e^{-\gamma L_0(\tau)} =\tfrac{b-a}{b-a-2\gamma ab}$. Interpretation of the other relation is analogous. 

 \section{Convergence in the operator norm}\label{cito}
Relations \eqref{3:1} and \eqref{3:2} combined state that 
\[ \grat \e^{tA} f(x) = f(a) \ell_* (x) + f(b) k_* (x) \]
but our arguments establish this convergence merely in pointwise sense. Our first  goal in this section is to show the following stronger result. 

\begin{proposition}\label{pro:cito1} Let $P$ be the operator mapping the space $\mathfrak C[a,b]$ into itself and given by $Pf=f(a)\ell_* + f(b)k_*$. There are constants $K\ge 0$ and $\kappa >0$ such that 
\[ \|\e^{tA} - P\|\le K \e^{-\kappa t}, \qquad t \ge 0.\] 
\end{proposition}
\begin{proof} Let $\mathfrak C_0(a,b)\coloneqq \{f \in \mathfrak C[a,b]\colon f(a)=f(b)=0\}$ and let $B$ be the operator in this space defined as follows: $Bf=\frac 12 f''$ on the  domain composed of twice continuously differentiable functions $f\in \mathfrak C_0(a,b)$ such that $f''$ belongs to $\mathfrak C_0(a,b)$. It is well-known (see e.g. \cite{knigazcup}*{p. 181}) that $B$ generates a Feller semigroup (the related process is a Brownian motion on the interval $(a,b)$, killed upon first hitting either end), and that there are constants $K_0\ge 0$ and $\kappa >0$ such that 
\begin{equation*}\|\e^{tB}\|\le K_0 \e^{-\kappa t}, \qquad t \ge 0. \end{equation*}

Also, we recall that $\rez{A}=\rla$ where $\rla $ is given in \eqref{2:4}, and that $\rla g(a)=\rla g(b) =0$ provided that $g(a)=g(b)=0$; indeed, $h(a)=h(b)=0$ because $\ell (b)=0$ and $k(a)=0$. It follows that $\mathfrak C_0(a,b)$ is left invariant by $\sem{A}$. Now, as restricted to this space, $\sem{A}$ describes a Brownian motion which can be killed either  when it reaches one of interval's ends or when its local time at $0$ exceeds an exponential time. Since the process generated by $B$  can be killed only in the first of these ways, 
\begin{equation*}  \|\e^{tA}f \| \le \|\e^{tA}|f|\| \le  \|\e^{tB}|f|\|  \le K_0\e^{-\kappa t}\|f\|, \qquad f\in \mathfrak C_0(a,b).\end{equation*}

Finally, both $k_*$ and $\ell_*$ are easily checked to belong to the kernel of $A$. It follows that $\e^{tA}k_*=k_*$ and $\e^{tA}\ell_*=\ell_*$ for all $t\ge 0$,  and thus for any $f\in \mathfrak C[a,b]$ we can write 
\[ \|\e^{tA}f -Pf\| = \|\e^{tA}(f -Pf)\|  \le K_0\e^{-\kappa t}\|f-Pf\| \le (1+\|P\|)K_0\e^{-\kappa t}\|f\|,\]  
because $Pk_*=k_*$ and $P\ell_*=\ell_*$ (in fact, $P$ is a projection on the subspace spanned by $k_*$ and $\ell_*$) and $f-Pf$ belongs to $\mathfrak C_0(a,b)$. This shows the desired estimate with $K\coloneqq (1+\|P\|)K_0$.
 \end{proof}

We complete the paper with a couple of words on the special case where $a=-b$. Our point is that then the semigroup $\sem{A}$ is of very special form, being build of two simpler semigroups, and this sheds some extra light on Proposition \ref{pro:cito1}.

To begin with, we consider the following abstract situation. Suppose that a Banach space $\mathfrak C$ is the direct sum of two of its closed subspaces: $\mathfrak C = \mathfrak C_1 \oplus \mathfrak C_2$. In particular, for each $f\in \mathfrak C$ there are unique vectors $Q_1f\in \mathfrak C_1$ and $Q_2 f \in \mathfrak C_2$ such that $f= Q_1f + Q_2 f$. Suppose furthermore that the subspaces $\mathfrak C_i$ are isomorphic to Banach spaces $\mathfrak B_i$, so that there are isomorphisms $J_i\colon \mathfrak C_i \to \mathfrak B_i$, $i=1,2$.   In this scenario, given two semigroups $\rod{S_1(t)}$ and $\rod{S_2(t)}$, defined in $\mathfrak C_1$ and $\mathfrak C_2$, respectively, 
one can define a semigroup $\rod{S(t)}$ in $\mathfrak C$ by the formula 
\[ S(t) \coloneqq \sum_{i=1}^2 J_i^{-1} S_i(t) J_i Q_i, \qquad t \ge 0.  \]
Indeed, the family $\rod{S(t)}$ is obviously strongly continuous, and to prove the semigroup property it suffices to use the relations
\begin{equation}\label{relations} Q_i J_i^{-1} = J_i^{-1}, \quad i=1,2 \mquad{ and } Q_1J_2^{-1} = Q_2 J_1^{-1} = 0.\end{equation}
The following lemma provides a natural characterization of the generator of $\rod{S(t)}$.

\begin{lemma}The domain of the generator, say, $G$, of $\rod{S(t)}$ is given by 
\[ \dom{G}= \{f\in \mathfrak C\colon J_i Q_i f \in \dom{G_i}, i =1,2\}\]
where $G_i$ is the generator of $\rod{S_i(t)}, i=1,2$. Also, for $f \in \dom{G}$, $Gf = \sum_{i=1}^2 J_i^{-1} G_i J_i Q_i$. \end{lemma}
\begin{proof} If $\lim_{t\to 0+}\frac 1t (S(t)f-f) $ exists for some $f$, then, by \eqref{relations}, so do the limits $\lim_{t\to 0+}\frac 1t J_i^{-1} (S_i(t) J_i Q_i-J_iQ_if), i=1,2.$ Since, however, $J_i$s are isomorphisms, this means that $J_iQ_if$ belongs to $\dom{G_i}, i=1,2$. Vice versa, 
if $J_i Q_i f \in \dom{G_i}, i =1,2$ then a straightforward calculation shows that the limit $\lim_{t\to 0+}\frac 1t (S(t)f-f) $ exists, and equals $\sum_{i=1}^2 J_i^{-1} G_i J_i Q_i$. \end{proof}

Returning to the semigroup $\sem{A}$, we note that $\mathfrak C[-b,b]$ is the direct sum of its subspaces of odd and even functions, denoted in what follows $\mathfrak C_1 $ and $\mathfrak C_2$, respectively, to comply with the notation introduced above. Moreover, the natural projection $Q_1$ maps a function $f\in \mathfrak C[-b,b]$ to its odd part $Q_1f$ given by $Q_1f (x) \coloneqq \frac 12 (f(x) -f(-x)), x \in [-b,b]$. Similarly, the projection $Q_2$ maps an $f$ to its even part defined by $Q_2f(x) \coloneqq \frac 12 (f(x) + f(-x)), x \in [-b,b]$.  Furthermore, $\mathfrak C_1$ is (isometrically) isomorphic to the space $\mathfrak C_0(0,b]$ of continuous functions on $[0,b]$ that vanish at $x=0$; the natural isomorphism $J_1\colon \mathfrak C_1\to \mathfrak C_0(0,b]$ is given by $J_1f(x) = f(x), x \in [0,b]$. Similarly, $\mathfrak C_2$ is (isometrically) isomorphic to the space $\mathfrak C[0,b]$ of continuous functions on $[0,b]$; the natural isomorphism $J_2\colon \mathfrak C_2\to \mathfrak C[0,b]$ is given by $J_2f(x) = f(x), x \in [0,b]$.

To continue, let $G_1$ be the operator defined in $\mathfrak C_0(0,b]$ as follows: its domain is composed of twice continuously differentiable functions on $[0,b]$ satisfying $f(0)=f''(0)=f''(b)=0$, and for such $f$ we let $G_1 f= \frac 12 f''$. Also, let the domain of $G_2$ be the set of $f\in \mathfrak C[0,b]$ such that $f'(0)=\gamma f(0)$ and $f''(b)=0$, and for $f$ in this domain let $G_2f= \frac 12f''.$   
\begin{proposition}The semigroup generated by $A$ has the following representation:
\begin{equation}\label{blokowa} \e^{tA} \coloneqq \sum_{i=1}^2 J_i^{-1}\e^{tG_i} J_i Q_i, \qquad t \ge 0.  \end{equation}\end{proposition}
\begin{proof} We know that the right-hand side defines a semigroup, and so our task is to show that this semigroup is generated by $A$. Let $G$ be this semigroup's generator. Since $A$ is a generator also and no generator can be a proper extension of another generator, it suffices to show that $G$ extends $A$. 

So, let $f$ belong to $\dom{A}$. Then, $J_1Q_1f$ is twice continuously differentiable on $[0,b]$ as the diference  of two twice continuously differentiable functions,  we clearly have $(J_1Q_1f)(0)=0$, condition $f''(0+) = f''(0-)$ implies $(J_1Q_1f)''(0)=0$ and $f''(-b)=f''(b) =0$ implies $(J_1Q_1f)''(b)=0$. This means that $J_1Q_1f$ belongs to $\dom{G_1}$. Similarly $J_2Q_2f$ is twice continuously differentiable and conditions \eqref{1:4} imply that $(J_2Q_2f)'(0)= \gamma (J_2Q_2f)(0)$ and $(J_2Q_2f)''(b)=0$, proving that $J_2Q_2f$ belongs to $\dom{G_2}$. Finally, it is easy to check that $Gf =\frac 12 f''=Af$, completing the proof. 
\end{proof}

Fundamentally, \eqref{blokowa} reveals that $\sem{A}$ leaves the subspaces of odd and even functions invariant. Since we also can identify the images of its subspace semigroups, this sheds additional light at Proposition \ref{pro:cito1}. For, the semigroup generated by $G_1$ describes the Brownian motion on $(0,b]$ killed at $0$ and stopped at $b$.  Hence, it is easy to first guess and then prove, using standard techniques of semigroup theory that, in the operator norm and with exponential speed,
\[ \grat \e^{tG_1} f= f(b) k_1, \qquad f \in \mathfrak C_0(0,b], \] 
where $k_1(x) = \frac {x}{b}$ is the probability that the process starting at $x$ will be eventually captured at $b$. Similarly, 
the semigroup generated by $G_2$ describes the Brownian motion on $[0,b]$ with elastic boundary at $0$, stopped upon hitting $b$ for the first time, and thus 
\[ \grat \e^{tG_2} f= f(b) k_2, \qquad f \in \mathfrak C[0,b], \] 
where $k_2(x) = \frac {1+\gamma x}{1+\gamma b}, x\in [0,b].$ 
Formula \eqref{blokowa} implies therefore 
\begin{align*} \grat \e^{tA} f &= \tfrac{f(b) -f(-b)}2 J^{-1}_1 k_1 +  \tfrac{f(b) +f(-b)}2 J^{-1}_2 k_2 \\
&= \tfrac 12 f(b) (J^{-1}_1 k_1 +  J^{-1}_2 k_2)  +  \tfrac 12 f(-b) (J^{-1}_2 k_2 -  J^{-1}_1 k_1),\end{align*}
and a little calculation shows that 
\begin{align*}
\tfrac 12 (J^{-1}_1 k_1 +  J^{-1}_2 k_2) (x) &= \tfrac {b+b\gamma |x| + x + b\gamma x}{2b(1+\gamma b)} = k_*(x), \\
\tfrac 12 (J^{-1}_2 k_2 -  J^{-1}_1 k_1) (x) &= \tfrac {b+b\gamma |x| - x - b\gamma x}{2b(1+\gamma b)} = \ell_*(x), \qquad x \in [-b,b],\end{align*}
establishing Proposition \ref{pro:cito1} once again.

Finally, it can be argued that $G_1$ and $G_2$ generate cosine families. A~formula analogous to \eqref{blokowa} proves that so does $A$. 

 
\bf Acknowledgment. \rm The authors would like to thank the Isaac Newton Institute for Mathematical Sciences, Cambridge, for support and hospitality during the programme \emph{Stochastic systems for anomalous diffusion}, where work on this paper was undertaken. This work was supported by EPSRC grant EP/Z000580/1. 
A.Pilipenko thanks  the Swiss National Science Foundation for partial support  of the paper (grants No IZRIZ0\_226875, No 200020\_200400, No. 200020\_192129).

\vspace{-0.6cm}
\bibliographystyle{plain}

\begin{bibdiv}
\begin{biblist}

\bib{abhn}{book}{
      author={Arendt, W.},
      author={Batty, C. J.~K.},
      author={Hieber, M.},
      author={Neubrander, F.},
       title={Vector-{V}alued {L}aplace {T}ransforms and {C}auchy {P}roblems},
   publisher={Birkh{\"a}user},
     address={Basel},
        date={2001},
}

\bib{bass}{book}{
      author={Bass, R.~F.},
       title={Stochastic processes},
      series={Cambridge Series in Statistical and Probabilistic Mathematics},
   publisher={Cambridge University Press, Cambridge},
        date={2011},
      volume={33},
        ISBN={978-1-107-00800-7},
         url={https://doi.org/10.1017/CBO9780511997044},
}

\bib{abielecki}{article}{
      author={Bielecki, A.},
       title={Une remarque sur la m\'ethode de {B}anach--{C}acciopo\-li--{T}ikhonov},
        date={1956},
     journal={Bull. Polish Acad. Sci.},
      volume={4},
       pages={261\ndash 268},
}

\bib{kniga}{book}{
      author={Bobrowski, A.},
       title={Functional {A}nalysis for {P}robability and {S}tochastic {P}rocesses. {A}n {I}ntroduction},
   publisher={Cambridge University Press, Cambridge},
        date={2005},
        ISBN={978-0-521-83166-6; 978-0-521-53937-1; 0-521-53937-4},
         url={http://dx.doi.org/10.1017/CBO9780511614583},
}

\bib{knigazcup}{book}{
      author={Bobrowski, A.},
       title={Convergence of {O}ne-{P}arameter {O}perator {S}emigroups. {I}n {M}odels of {M}athematical {B}iology and {E}lsewhere},
   publisher={Cambridge University Press, Cambridge},
        date={2016},
}

\bib{jedenipol}{book}{
      author={Bobrowski, A.},
       title={Functional {A}nalysis {R}evisited. {A}n {E}ssay on {C}ompleteness},
   publisher={Cambridge University Press, Cambridge},
        date={2024},
}

\bib{zmarkusem2b}{article}{
      author={Bobrowski, A.},
      author={Kunze, M.},
       title={Irregular convergence of mild solutions of semilinear equations},
        date={2019},
        ISSN={0022-247X},
     journal={J. Math. Anal. Appl.},
      volume={472},
      number={2},
       pages={1401\ndash 1419},
         url={https://doi.org/10.1016/j.jmaa.2018.11.082},
}

\bib{borodin}{book}{
      author={Borodin, A.~N.},
      author={Salminen, P.},
       title={Handbook of {B}rownian motion---facts and formulae},
     edition={Second Edition},
      series={Probability and its Applications},
   publisher={Birkh\"auser Verlag, Basel},
        date={2002},
        ISBN={3-7643-6705-9},
         url={https://doi.org/10.1007/978-3-0348-8163-0},
}

\bib{edwards}{book}{
      author={Edwards, R.~E.},
       title={Functional {A}nalysis. {T}heory and {A}pplications},
   publisher={Dover Publications},
        date={1995},
}

\bib{ethier}{book}{
      author={Ethier, S.~N.},
      author={Kurtz, T.~G.},
       title={{M}arkov {P}rocesses. {C}haracterization and {C}onvergence},
   publisher={Wiley},
     address={New York},
        date={1986},
}

\bib{fellera3}{article}{
      author={Feller, W.},
       title={The parabolic differential equations and the associated semi-groups of transformations},
        date={1952},
     journal={Ann. Math.},
      volume={55},
       pages={468\ndash 519},
}

\bib{fellera4}{article}{
      author={Feller, W.},
       title={Diffusion processes in one dimension},
        date={1954},
     journal={Trans. Amer. Math. Soc.},
      volume={77},
      number={1},
       pages={1\ndash 31},
}

\bib{goldstein}{book}{
      author={Goldstein, J.~A.},
       title={Semigroups of {L}inear {O}perators and {A}pplications},
   publisher={Oxford University Press},
     address={New York},
        date={1985},
}

\bib{hillebook}{book}{
      author={Hille, E.},
       title={Functional {A}nalysis and {S}emi-{G}roups},
      series={American Mathematical Society Colloquium Publications, vol. 31},
   publisher={American Mathematical Society, New York},
        date={1948},
}

\bib{ito}{book}{
      author={It{\^o}, K.},
      author={P., {McKean, Jr.}~H.},
       title={Diffusion {P}rocesses and {T}heir {S}ample {P}aths},
   publisher={Springer},
     address={Berlin},
        date={1996},
        note={Repr. of the 1974 ed.},
}

\bib{kallenbergnew}{book}{
      author={Kallenberg, O.},
       title={Foundations of {M}odern {P}robability},
     edition={2},
   publisher={Springer},
        date={2002},
}

\bib{karatzas}{book}{
      author={Karatzas, I.},
      author={Shreve, S.~E.},
       title={Brownian {M}otion and {S}tochastic {C}alculus},
   publisher={Springer},
     address={New York},
        date={1991},
        ISBN={0-387-97655-8},
}

\bib{rosen}{book}{
      author={Marcus, M.~B.},
      author={Rosen, J.},
       title={Markov {P}rocesses, {G}aussian {P}rocesses, and {L}ocal {T}imes},
      series={Cambridge Studies in Advanced Mathematics},
   publisher={Cambridge University Press},
        date={2006},
}

\bib{pazy}{book}{
      author={Pazy, A.},
       title={Semigroups of {L}inear {O}perators and {A}pplications to {P}artial {D}ifferential {E}quations},
   publisher={Springer},
        date={1983},
}

\end{biblist}
\end{bibdiv}
\def\cprime{$'$}\def\cprime{$'$}\def\cprime{$'$}\def\polhk#1{\setbox0=\hbox{#1}{\ooalign{\hidewidth \lower1.5ex\hbox{`}\hidewidth\crcr\unhbox0}}}\ifx \cedla \undefined \let \cedla = \c \fi\ifx \cyr \undefined \let \cyr = \relax \fi\ifx \cprime \undefined \def \cprime {$\mathsurround=0pt '$}\fi\ifx \prime \undefined \def \prime {'} \fi\def\polhk#1{\setbox0=\hbox{#1}{\ooalign{\hidewidth \lower1.5ex\hbox{`}\hidewidth\crcr\unhbox0}}}

\end{document}